\documentclass[11pt]{article}

\usepackage[T1]{fontenc}
\usepackage[utf8]{inputenc}
\usepackage{lmodern}
\usepackage[margin=1in]{geometry}
\usepackage{amsmath,amssymb,amsthm}
\usepackage{booktabs}
\usepackage{graphicx}
\usepackage{float}
\usepackage{flafter}
\usepackage{needspace}
\usepackage{microtype}
\usepackage{xcolor}
\definecolor{linkblue}{RGB}{0,85,170}
\usepackage[colorlinks=true,linkcolor=linkblue,citecolor=linkblue,urlcolor=linkblue]{hyperref}
\hypersetup{
  pdftitle={Rayleigh's conjecture for vibrating plates: simpler planar proof and nonlinear extensions},
  pdfauthor={Dorin Bucur and Richard S. Laugesen},
  pdfkeywords={Rayleigh's conjecture, clamped plate, isoperimetric inequality, nonlinear eigenvalue, signed rearrangement}
}

\newtheorem{theorem}{Theorem}
\newtheorem{proposition}[theorem]{Proposition}
\newtheorem{lemma}[theorem]{Lemma}
\newtheorem{conjecture}[theorem]{Conjecture}
\makeatletter
\let\c@table\c@figure
\makeatother

\newcommand{\R}{\mathbb R}
\newcommand{\norm}[1]{\left\lVert#1\right\rVert}

\title{Rayleigh's conjecture for vibrating plates: simpler planar proof and nonlinear extensions}
\author{Dorin Bucur\thanks{Universit\'e Savoie Mont Blanc, CNRS, LAMA,
73000 Chamb\'ery, France.
\href{mailto:Dorin.Bucur@univ-smb.fr}{\texttt{<Dorin.Bucur@univ-smb.fr>}}.}
\and Richard S. Laugesen\thanks{Dept. of Mathematics, Univ.\ of Illinois Urbana--Champaign, Urbana, IL 61801, USA.
\href{mailto:laugesen@illinois.edu}{\texttt{<laugesen@illinois.edu>}}.}}

\begin{document}
\maketitle
\begingroup
\renewcommand{\thefootnote}{}
\footnotetext[0]{%
\textit{2020 Mathematics Subject Classification.}
Primary 35P15; Secondary 35P30, 35J35.\par
\textit{Keywords.} Rayleigh's conjecture, clamped plate,
isoperimetric inequality, nonlinear eigenvalue, signed rearrangement.}
\endgroup

\begin{abstract}
We give a simpler proof of Nadirashvili's theorem (Rayleigh's conjecture)  
that among vibrating clamped plates of prescribed
area in the plane, the disk has the lowest fundamental tone. The proof uses a one-shot Talenti type comparison for sign-changing solutions of the Poisson equation with mean-zero data, established in a companion paper. The comparison considers just two rearranged sources, rather than optimizing over all admissible sources as in the traditional two-ball technique. This ``paired comparison'' simplification yields immediate nonlinear extensions: the disk minimizes the clamped quotient
$\norm{\Delta u}_q/\norm{u}_p$ among bounded open sets of prescribed area, for every $1\le p,q\le\infty$. Equality is characterized for $q>1$, while exceptional behavior at $q=1$ is identified. In higher dimensions, we prove ball minimality for $p=\infty$ and conjecture minimality for all $p,q$ within the compact embedding range. We also establish the corresponding interval inequality in one dimension.
\end{abstract}

\section{Introduction and main results}\label{sec:introduction}

Which vibrating clamped plate of prescribed area has the lowest
fundamental tone? Rayleigh conjectured the disk is the minimizer,
which Nadirashvili~\cite{Nadirashvili1995} then proved more than a century
later. We give a significantly simpler planar proof and show
that the disk remains optimal when the $L^2$ norms in the
variational formulation are replaced by $L^p$ and $L^q$
norms, for every $1\le p,q\le\infty$.

Nadirashvili's proof and the proof of Ashbaugh and
Benguria~\cite{AshbaughBenguria1995} in dimensions two and three
use the two-ball method: Talenti's elliptic rearrangement~\cite{Talenti1976} is applied
on the nodal domains, followed by minimization over all admissible
sources on the coupled balls. Our \emph{paired comparison}
theorem~\cite{BL2026} reduces this optimization to just two source choices:
the signed rearrangements of the original mean-zero source $f=-\Delta u$
and its negative, each on the equal-area disk. One of the
resulting potentials has $L^p$ norm at least as large as for $u$, while the $L^q$ norm of the Laplacian
 is preserved. This simplified approach applies to 
all $p$ and $q$, without using coupled nonlinear Euler--Lagrange equations that in the nonlinear framework would raise significant difficulties.

\newpage
\subsection*{Definitions}
For a nonempty bounded open set $\Omega\subset\R^n$, write
$\Omega^*$ for the centered ball with the same volume, $|\Omega^*|=|\Omega|$.
The Sobolev space $W_0^{2,q}(\Omega)$ is the
closure of $C_c^\infty(\Omega)$ in $W^{2,q}(\Omega)$, for
$1\le q<\infty$. Write $H_0^k=W_0^{k,2}$ and abbreviate
$\norm{\cdot}_{L^s(\Omega)}$ to $\norm{\cdot}_s$ when the domain is clear. {
For $q=\infty$, we use the ``clamped space'' of Katzourakis and
Parini~\cite[Eq.~(1.2)]{KatzourakisParini2017}:
\[
 W_C^{2,\infty}(\Omega)
 =\left\{u\in\bigcap_{1<r<\infty}W_0^{2,r}(\Omega):
                    \Delta u\in L^\infty(\Omega)\right\}.
\]
\Needspace{7\baselineskip}
Define nonlinear ``clamped plate'' eigenvalues 
\begin{equation}\label{eq:gamma}
 \begin{aligned}
 \Gamma_{p,q}(\Omega)
 &=\inf_{0\ne u\in W_0^{2,q}(\Omega)\cap L^p(\Omega)}
   \frac{\norm{\Delta u}_{L^q(\Omega)}}{\norm{u}_{L^p(\Omega)}},
       && \quad 1\le q<\infty,\\
 \Gamma_{p,\infty}(\Omega)
 &=\inf_{0\ne u\in W_C^{2,\infty}(\Omega)}
   \frac{\norm{\Delta u}_{L^\infty(\Omega)}}{\norm{u}_{L^p(\Omega)}} , && \quad q=\infty ,
 \end{aligned}
\end{equation}
for $1\le p\le\infty$. The $L^p$ restriction in this definition simply serves to exclude infinite denominators. It is automatically satisfied in dimensions one and two,
and for $q=\infty$ in every
dimension, by Sobolev embedding (Section~\ref{sec:attainment}). 
The scaling law for the eigenvalues is
\begin{equation}\label{eq:scaling}
 \Gamma_{p,q}(t\Omega)
 =t^{n(1/q-1/p)-2}\Gamma_{p,q}(\Omega) , \qquad t>0 .
\end{equation}
Later we show the infima in \eqref{eq:gamma} can
equivalently be taken over smooth functions compactly supported in $\Omega$, by Lemma~\ref{lem:smooth-trials}.
}

The special case $(p,q)=(2,2)$ gives the classical first clamped plate eigenvalue 
\[
 \Gamma_{2,2}(\Omega)^2
 =\inf_{0\ne u\in H_0^2(\Omega)}
 \frac{\int_\Omega|\Delta u|^2\,dx}
      {\int_\Omega|u|^2\,dx}.
\]
And with $(p,q)=(1,2)$, the quantity $\Gamma_{1,2}(\Omega)^{-2}$ describes the compliance of a plate under unit-magnitude loads; see the remark at the end of Section~\ref{sec:planarproof}.

\subsection*{Disk minimality for all $p,q$}

\begin{theorem}[Disk minimality for all exponents]\label{thm:disk}
If $\Omega\subset\R^2$ is a nonempty bounded open set and
$1\le p,q\le\infty$, then
\begin{equation}\label{eq:main}
 \Gamma_{p,q}(\Omega)\ge\Gamma_{p,q}(\Omega^*).
\end{equation}

When $q>1$, equality holds in \eqref{eq:main} if and only if
$\Omega$ is a disk.
When $q=1$ and $1\le p<\infty$, equality holds if $\Omega$
is a disk with a set of $H^1$-capacity zero removed.
When $q=1$ and $p=\infty$, both sides of \eqref{eq:main} equal zero.
\end{theorem}

Figure~\ref{fig:range} depicts the parameter square in $(1/p,1/q)$ coordinates. 
Section~\ref{sec:planarproof} proves inequality \eqref{eq:main} and
Sections~\ref{sec:planar-equality} and~\ref{sec:q1-equality}
treat equality for $q>1$ and $q=1$, respectively. 
We do not know whether the
equality statement when $q=1$ and $1 \leq p<\infty$ can be strengthened to ``if and only if''.

{
Several cases of the theorem were known previously, notably Rayleigh's vibrating clamped plate case $(p,q)=(2,2)$ proved by Nadirashvili~\cite{Nadirashvili1995}. Our proof of this case of Theorem~\ref{thm:disk} relies on $L^2$ paired comparison, the proof of which is particularly transparent \cite[Section 4]{BL2026}, hence our claim of a ``simpler planar proof'' for Rayleigh's conjecture.}

The top edge case $q=1$ in Figure~\ref{fig:range} with $1\le p<\infty$ follows by integrating
Fontana and Morpurgo's sharp rearrangement estimate and using their radial approximating sequence
\cite[Theorem~1]{FontanaMorpurgo2014}.
Parini, Ruf and Tarsi proved the $(p,q)=(1,1)$ inequality for $C^{1,1}$ domains
in their study of the clamped $1$-biharmonic problem
and its relaxation~\cite[Theorem~1.2 and Corollary~7.3]{PariniRufTarsi2015}.
At $(p,q)=(\infty,\infty)$, Katzourakis and Parini proved ball
minimality among $C^{1,1}$ domains \cite[Theorem~1.2]{KatzourakisParini2017}. 

On a disk, the eigenvalue infimum
is attained by a positive radial function $u$ with strictly decreasing
profile, when $1\le p\le\infty$ and $1<q\le\infty$, by Proposition~\ref{prop:attainment}(\ref{part:attainment-existence})
and the construction in Section~\ref{sec:planar-equality}.
For $1<q<\infty$, every minimizer is a nonzero constant multiple of this $u$; we omit
the proof of this uniqueness claim.

\paragraph{Remark (explicit values for the disk).}
The eigenvalue $\Gamma_{p,q}$ for a disk can be evaluated directly at every point around the edges of the parameter square. Those formulas are not needed in the paper and so we omit them. The corresponding interval formulas in one dimension are presented in Table \ref{tab:interval-values} below.

\begin{figure}[!p]
 \centering
 \includegraphics[scale=0.85,trim=0.944bp 0 -237.186bp 0]{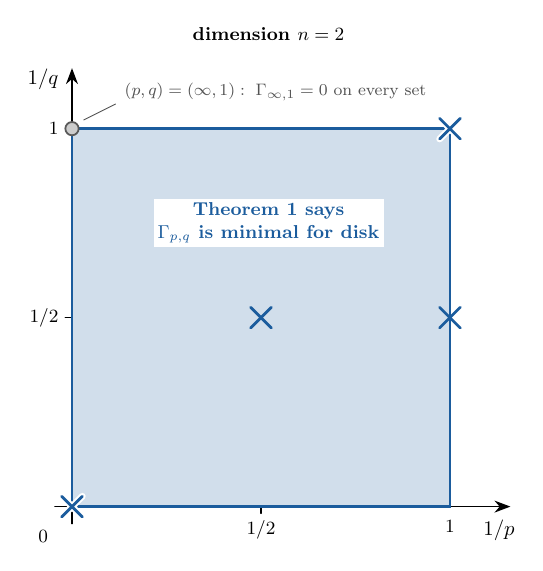}
 \vspace{-6pt}
 \caption{The parameter range for Theorem~\ref{thm:disk} in two dimensions, shown with $(1/p,1/q)$ coordinates.
 Crosses indicate earlier known cases: Nadirashvili at $(p,q)=(2,2)$,
 Ashbaugh--Bucur--Laugesen--Leylekian at $(1,2)$,
 Katzourakis--Parini at $(\infty,\infty)$ and
 Parini--Ruf--Tarsi at $(1,1)$. 
 The top edge $q=1$ follows from
 Fontana--Morpurgo. The gray dot in the top left 
 marks a special case where $\Gamma_{\infty,1}=0$.}
 \label{fig:range}
 \par
 \vspace{12pt}
 \includegraphics[scale=0.85]{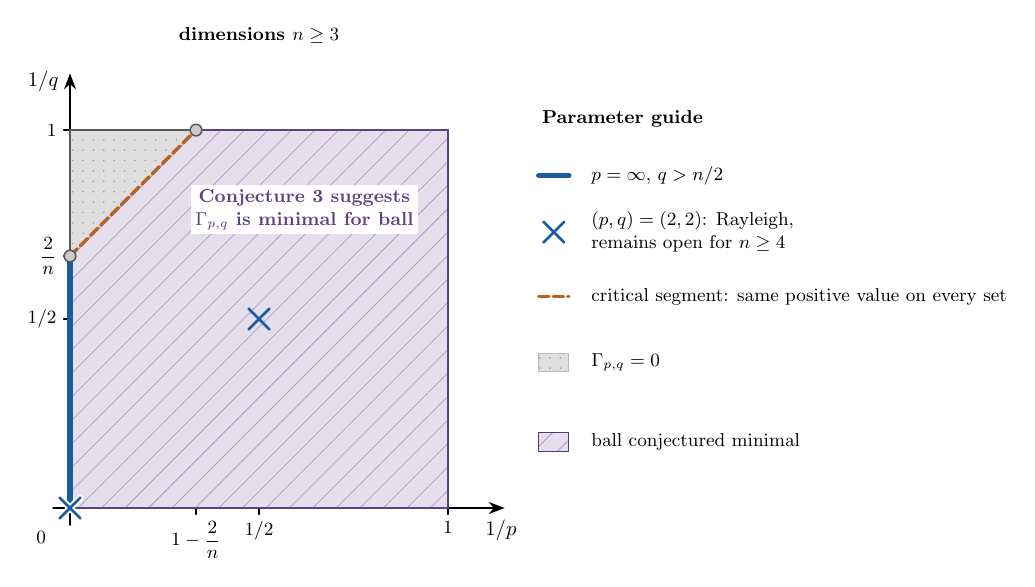}
 \vspace{-6pt}
 \caption{Dimensions $n \geq 3$. Conjecture \ref{conj:nd-gamma} is raised throughout the hatched parameter region, which is the compact embedding range $1/q<1/p+2/n$ (drawn to scale at
 $n=3$). Theorem~\ref{thm:nd-infty} proves it along the solid blue
 segment, where $p=\infty$. Crosses mark known results by Ashbaugh--Benguria at $(p,q)=(2,2)$ for dimension three, and
 Katzourakis--Parini at $(\infty,\infty)$ for all dimensions. 
 Along the orange dashed critical segment, $\Gamma_{p,q}$ takes a positive value independent
 of shape, while $\Gamma_{p,q}$ equals zero in the gray dotted region, including the outer edges and gray dot endpoints (Proposition~\ref{prop:nd-noncompact}).}
 \label{fig:three-range}
\end{figure}

\subsection*{Higher dimensional ball minimality: proved for $p=\infty$, conjectured in general}
\label{sec:higher-dimensions}

The paired comparison also proves ball minimality in higher
dimensions when $p=\infty$.

\begin{theorem}[Higher dimensional ball minimality, for $p=\infty$]
\label{thm:nd-infty}
Suppose $n\ge3$. If $\Omega\subset\R^n$ is nonempty, bounded
and open, and $n/2<q\le\infty$, then
\begin{equation}\label{eq:nd-infty}
 \Gamma_{\infty,q}(\Omega)\ge\Gamma_{\infty,q}(\Omega^*).
\end{equation}
Equality holds if and only if $\Omega$ is a ball.
\end{theorem}

Section~\ref{sec:all-dimensions} proves the inequality and
Section~\ref{sec:nd-equality} handles the equality case.

\begin{conjecture}[Higher dimensional ball minimality]
\label{conj:nd-gamma}
Suppose $n\ge3$. If $\Omega\subset\R^n$ is nonempty, bounded and open,
and $1\le p,q\le\infty$ satisfy
\[
 \frac1q<\frac1p+\frac2n,
\]
then
\[
 \Gamma_{p,q}(\Omega)\ge\Gamma_{p,q}(\Omega^*).
\]
\end{conjecture}

Figure~\ref{fig:three-range} summarizes the theorem and conjecture along with known special cases.
The assumption $1/q<1/p+2/n$ specifies the compact embedding
range (Section~\ref{sec:attainment}). Outside it,
$\Gamma_{p,q}$ is independent of the shape of the domain (Proposition~\ref{prop:nd-noncompact}).

The case $(p,q)=(2,2)$ is already settled when $n=3$ by Ashbaugh and Benguria's proof \cite{AshbaughBenguria1995} of Rayleigh's clamped plate conjecture. Rayleigh's conjecture remains open for $n\ge4$. 

\Needspace{12\baselineskip}
\subsection*{One dimensional interval minimality for all $p,q$}

In one dimension, $\Delta u$ means $u''$. Closing up the gaps between components of the domain 
proves interval minimality for all parameters $p$ and $q$ as shown in
Section~\ref{sec:one-dimension}, with strict inequality for disconnected sets later proved in Section~\ref{sec:interval-equality}.

\begin{proposition}[Interval minimality for all $p,q$]
\label{prop:interval}
Let $\Omega\subset\R$ be a nonempty bounded open set, so that $\Omega^*$ is an interval of the same length. If $1\le p,q\le\infty$ then 
\begin{equation}\label{eq:interval-comparison}
 \Gamma_{p,q}(\Omega)\ge\Gamma_{p,q}(\Omega^*).
\end{equation}
Equality holds if and only if $\Omega$ is an interval. 
\end{proposition}

The interval eigenvalue has a positive
minimizer $u$ that is symmetric about the midpoint and strictly decreasing
toward each endpoint, when $1\le p\le\infty$ and $1<q \leq \infty$. Every minimizer is a nonzero constant multiple
of $u$. We omit the proofs. For $1<p,q<\infty$, interval uniqueness also follows from
Benedikt~\cite[Theorem~8, Lemma~24 and Example~18]{Benedikt2004Simplicity}.
For $p=q>1$, Benedikt also proved that the first clamped
eigenfunction has one sign~\cite[Theorem~1.1]{Benedikt2004Discreteness};
midpoint symmetry then follows from reflection and simplicity.

\paragraph{Remark (explicit values for the interval).}
{
For an interval $I$ of length $L$, Table~\ref{tab:interval-values}
collects the four edge formulas, with proofs in
Appendix~\ref{app:interval-values}. Here $q'$ is the H\"older conjugate
exponent, including $1'=\infty$ and $\infty'=1$. We interpret
$(r+1)^{1/r}=1$ when $r=\infty$ and define constants $c_\infty=1$ and 
\[
 \begin{aligned}
 c_p & =2^{1+1/p}\left(\frac1{2p+1}
             +\int_0^1(2-t^2)^p\,dt\right)^{\!-1/p},
 && 1\le p<\infty ,\\
 m_r & =\left(\min_{0\le a\le1}\norm{t^2-a}_{L^r(0,1)}\right)^{\!-1},
 && 1\le r\le\infty.
 \end{aligned}
\]

\begin{table}[htbp]
\centering
\setlength{\abovecaptionskip}{8pt}
\begingroup
\renewcommand{\arraystretch}{1.6}
\begin{tabular}{@{}llcl@{}}
\toprule
Edge & Parameter range & $\Gamma_{p,q}(I)$ & Proof\\
\midrule
$p=\infty$ & $1\le q\le\infty$
 & $\displaystyle\frac{8(q'+1)^{1/q'}}{L^{2-1/q}}$
 & Appendix~\ref{app:interval-pinfty}\\
\addlinespace[4pt]
$q=1$ & $1\le p\le\infty$
 & $\displaystyle\frac{8(p+1)^{1/p}}{L^{1+1/p}}$
 & Appendix~\ref{app:interval-qone}\\
\addlinespace[4pt]
$q=\infty$ & $1\le p\le\infty$
 & $\displaystyle\frac{16c_p}{L^{2+1/p}}$
 & Appendix~\ref{app:interval-qinfty}\\
\addlinespace[4pt]
$p=1$ & $1\le q\le\infty$
 & $\displaystyle\frac{8m_{q'}}{L^{3-1/q}}$
 & Appendix~\ref{app:interval-pone}\\
\bottomrule
\end{tabular}
\endgroup
\par
\caption{Explicit edge values for an interval $I$ of length $L$. See discussion and references in the text.
\label{tab:interval-values}}
\end{table}

For $p=\infty$ and $1<q<\infty$, the formula was obtained by
Oshime~\cite[Theorem~14]{Oshime2008} and
reproved by Watanabe \emph{et al.}~\cite[Theorem~1.1]{WatanabeEtAl2009}.
Its endpoints give $\Gamma_{\infty,1}(I)=8/L$ and
$\Gamma_{\infty,\infty}(I)=16/L^2$.
For $q=\infty$, a minimizing function for every $p$ is the piecewise
quadratic profile of Katzourakis and
Parini~\cite[Remark~3.7]{KatzourakisParini2017}.
Parini--Ruf--Tarsi~\cite[Proposition~7.1]{PariniRufTarsi2015}
evaluated $\Gamma_{1,1}(I)$. 

For the $p=1$ edge, we adapt the interpolation argument underlying
the $q=2$ kernel proof of
Hindov--Nitzan--Olsen--Rydhe~\cite[Theorem~1, Sections~2.1--2.2]{HindovEtAl2025}.
Their value at $(p,q)=(1,2)$, lying halfway up the right side of the
parameter square, is $\Gamma_{1,2}(I)=12\sqrt{5}/L^{5/2}$. Hence with bending stiffness normalized to one, 
$\Gamma_{1,2}(I)^{-2}=L^5/720$ is the compliance of the clamped beam under uniform unit load. They also obtained higher order Sobolev embedding constants. 

For a final special case, a classical formula for the vibrating clamped beam says $\Gamma_{2,2}(I)=\beta_1^2/L^2$ where $\beta_1\approx4.73004$ is the least positive root of $\cos\beta_1\cosh\beta_1=1$.
}

\subsection*{Prior work on Rayleigh's conjecture}
\label{sec:literature}

The history of Rayleigh's conjecture centers around the difficulty of
sign-changing fundamental modes.
Szeg\H{o} proved the inequality under the additional assumption
that the fundamental mode has one sign~\cite{Szego1950,Szego1958}. Talenti
introduced the two-ball nodal domain rearrangement method to avoid that
assumption~\cite{Talenti1981}. Nadirashvili completed the planar
proof~\cite{Nadirashvili1995}, and Ashbaugh and Benguria established
the result in dimensions two and three~\cite{AshbaughBenguria1995}.
In dimensions four and higher, Ashbaugh and Laugesen obtained
quantitative lower bounds that approach the conjectured sharp
constant as the dimension increases~\cite{AshbaughLaugesen1996}.
Leylekian obtained sufficient conditions for Rayleigh's conjecture
in higher dimensions~\cite{Leylekian2024}.
Ashbaugh, Benguria and Laugesen~\cite{ABL1997} survey these plate
results and the related buckling problem.
For broader background, see Ashbaugh's survey~\cite{Ashbaugh2000}
on isoperimetric eigenvalue inequalities and the monograph of
Gazzola, Grunau and Sweers~\cite{GazzolaGrunauSweers2010} on
biharmonic problems and positivity.

Beyond Euclidean space, Krist\'aly proved sharp Rayleigh comparisons
in hyperbolic space for plates of sufficiently small volume in
dimensions two and three~\cite{Kristaly2020}, and for spherical
plates of arbitrary size in those dimensions or sufficiently large
size in higher dimensions~\cite{Kristaly2022}. Krist\'aly and Mondino obtained related
inequalities, with rigidity and stability results, on certain nonsmooth
spaces with synthetic nonnegative Ricci curvature~\cite{KristalyMondino2025}.
Chasman and Langford proved a nonsharp analogue in Gauss space for the
bi-Hermite operator~\cite{ChasmanLangford2016}.

\Needspace{22\baselineskip}
\subsection*{Positivity and attainment of $\Gamma_{p,q}$}\label{sec:attainment-statement}

The equality proofs use minimizing functions that are supplied by the next proposition, 
when $1<q\le\infty$. Figure~\ref{fig:attainment-range} summarizes the planar case.
Within the compact embedding range, the infimum at $q=1$
is positive but not attained.

{
\begin{proposition}[Attainment in the compact range]\label{prop:attainment}
Let $n\ge1$ and suppose $\Omega\subset\R^n$ is nonempty, bounded and open. If $1\le p,q\le\infty$ satisfy 
\[
 \frac1q<\frac1p+\frac2n
\]
then the following conclusions hold. 
\begin{enumerate}
\setlength{\itemsep}{2pt}
\setlength{\parsep}{0pt}
\renewcommand{\theenumi}{\alph{enumi}}
\renewcommand{\labelenumi}{(\theenumi)}
\item\label{part:attainment-compactness}
The embedding of $W_0^{2,q}(\Omega)$ into $L^p(\Omega)$ is compact  
when $q<\infty$, and when $q=\infty$ the embedding of $W_C^{2,\infty}(\Omega)$ (equipped with
the norm $\norm{\Delta u}_\infty$) into $L^p(\Omega)$ is compact.
Also, $\Gamma_{p,q}(\Omega)>0$.
\item\label{part:attainment-existence}
If $1<q\le\infty$ then the infimum for $\Gamma_{p,q}(\Omega)$ in \eqref{eq:gamma}
is attained by a nonzero function.
\item\label{part:attainment-finite-endpoints}
If $q=1$ and $1\le p<\infty$ then the infimum
is not attained.
\item\label{part:attainment-one-dimensional-q1}
If $n=1$, $p=\infty$ and $q=1$,
then the infimum is not attained.
\end{enumerate}
\end{proposition}
}

Section~\ref{sec:attainment} proves all four parts of the proposition.

\begin{figure}[!t]
 \centering
 \includegraphics[scale=0.85]{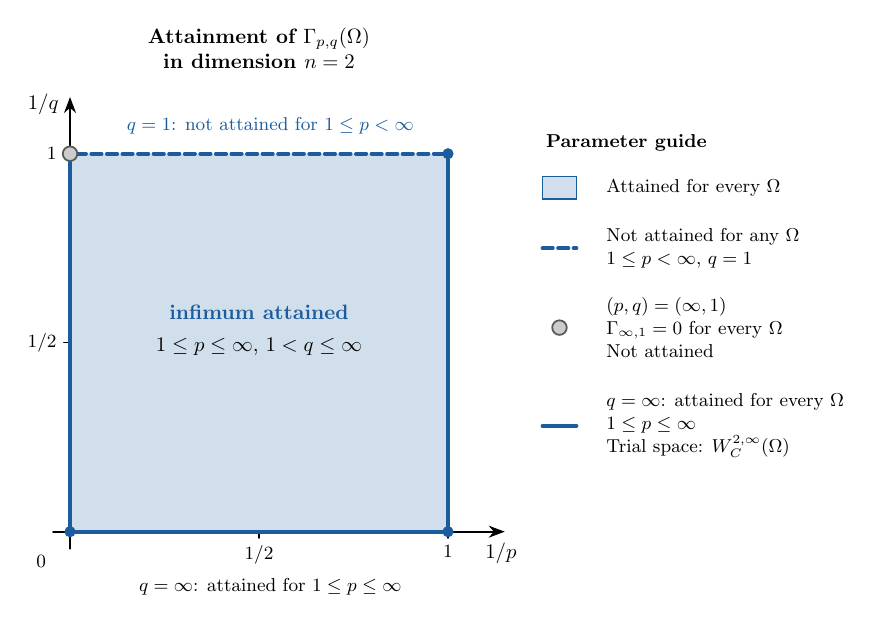}
 \vspace{-6pt}
 \caption{Dimension $n=2$. Attainment of the infimum defining $\Gamma_{p,q}(\Omega)$ for a nonempty bounded open set
 $\Omega\subset\R^2$, as justified by 
 Proposition~\ref{prop:attainment}(\ref{part:attainment-existence})(\ref{part:attainment-finite-endpoints}).
 The gray corner has an unattained zero infimum
 (Section~\ref{sec:q1-equality}).}
 \label{fig:attainment-range}
\end{figure}

\Needspace{10\baselineskip}
\section{Reduction to smooth trial functions}\label{sec:smooth-trials}

We reduce the inequality proofs to smooth compactly supported trial
functions. First we record two estimates used later for positivity
and attainment. The whole-space Calder\'on--Zygmund estimate
\cite[Chapter~9]{GilbargTrudinger2001} is
\begin{equation}\label{eq:whole-space-cz}
 \norm{D^2\phi}_{L^r(\R^n)}
 \le C_{n,r}\norm{\Delta\phi}_{L^r(\R^n)},
 \qquad \phi\in C_c^\infty(\R^n),\quad 1<r<\infty.
\end{equation}
For a nonempty bounded open set $\Omega\subset\R^n$ and
$u\in W_C^{2,\infty}(\Omega)$, zero extension, finite-exponent
density, Poincar\'e inequalities on a fixed enclosing ball and
H\"older's inequality give
\begin{equation}\label{eq:wc-bound}
 \norm{u}_{W^{2,r}(\Omega)}
 \le C_{\Omega,n,r}\norm{\Delta u}_\infty,\qquad 1<r<\infty.
\end{equation}

{
\begin{lemma}[Equivalent trial space]\label{lem:smooth-trials}
If $\Omega\subset\R^n, n\ge1$, is nonempty, bounded and open,
and $1\le p,q\le\infty$, then
\begin{equation}\label{eq:smooth-trials}
 \Gamma_{p,q}(\Omega)
 =\inf_{0\ne u\in C_c^\infty(\Omega)}
       \frac{\norm{\Delta u}_{L^q(\Omega)}}{\norm{u}_{L^p(\Omega)}}.
\end{equation}
Moreover, when $q=\infty$ every $u\in W_C^{2,\infty}(\Omega)$ has approximants
$u_j\in C_c^\infty(\Omega)$ such that
\begin{equation}\label{eq:smooth-approximation}
 u_j\longrightarrow u\quad\hbox{uniformly},\qquad
 \limsup_j\norm{\Delta u_j}_\infty\le\norm{\Delta u}_\infty.
\end{equation}
\end{lemma}

\begin{proof}
The ``$\leq$'' direction of \eqref{eq:smooth-trials} follows by inclusion of $C_c^\infty(\Omega)$ into the trial space.

For the reverse direction, first suppose $1\le q<\infty$. Take
$0\ne u\in W_0^{2,q}(\Omega)\cap L^p(\Omega)$ and choose
$u_j\in C_c^\infty(\Omega)$ converging to $u$ in $W^{2,q}$ and
almost everywhere. Fatou's lemma, or lower semicontinuity of the
essential supremum when $p=\infty$, gives
\[
 0<\norm{u}_p\le\liminf_j\norm{u_j}_p,
 \qquad
 \norm{\Delta u_j}_q\longrightarrow\norm{\Delta u}_q .
\]
Consequently, $\norm{\Delta u}_q/\norm{u}_p$ is greater than or equal to the right side of \eqref{eq:smooth-trials}. Taking the infimum over $u$ proves the ``$\geq$'' direction of \eqref{eq:smooth-trials}. 

For $q=\infty$, it suffices to prove \eqref{eq:smooth-approximation},
since uniform convergence gives convergence in every $L^p$ norm.
A related smooth approximation is constructed by Katzourakis and
Moser~\cite[proof of Lemma~2.2]{KatzourakisMoser2023} under a Lipschitz
assumption on $\Omega$; the proof below requires no boundary regularity.
Let $0\ne u\in W_C^{2,\infty}(\Omega)$ and put
$A=\norm{\Delta u}_\infty>0$.
The zero extension, still
denoted $u$, lies in $W^{2,r}(\R^n)$ with $\Delta u=0$ almost
everywhere on $\Omega^c$, for every $1<r<\infty$.
Choosing $r>n$, Sobolev embedding gives
$u\in C^1(\R^n)$ with $u=\nabla u=0$ at every point of $\Omega^c$.

We first estimate the boundary decay of $u$.
Choose a smooth regularized distance $d\in C^\infty(\Omega)$,
comparable to $\operatorname{dist}(x,\Omega^c)$ and satisfying
$|\nabla d|\le C$ and $|D^2d|\le C/d$; this construction requires
no boundary regularity~\cite[Lemma~1.1 and Corollary~1.2]{Lieberman1985}.
Fix $D>4\operatorname{diam}\Omega$ large enough that $d<D/e$
and put $\Lambda(x)=1+\log(D/d(x))$. 
For $n\ge2$, choose $x_0\in\Omega$ and consider the zero
extension on $B_D(x_0)$. Comparing $u$ and $-u$ with the torsion function 
$A(D^2-|x-x_0|^2)/(2n)$ on that ball gives $\norm{u}_\infty\le AD^2/(2n)$.
Since $\overline\Omega\subset B_{D/2}(x_0)$, the logarithmic
gradient estimate~\cite[Theorem~3.9, p.~41]{GilbargTrudinger2001}
(also valid for weak solutions with bounded right-hand side) can be applied
here to the zero extension on the larger ball $B_D(x_0)$, yielding
\[
 |\nabla u(x)-\nabla u(y)|
 \le CA|x-y|\bigl(1+\log(D/|x-y|)\bigr),
 \qquad x,y\in\overline\Omega,\quad x\ne y.
\]
Fix $x\in\Omega$, choose $y\in\Omega^c$ nearest to $x$ and put
$a=|x-y|$. Since $\nabla u(y)=0$, this estimate gives
$|\nabla u(x)|\le CAa(1+\log(D/a))$.
In dimension one, $|u''|\le A$ gives the stronger bound
$|u'(x)|\le Aa$.
Using also $u(y)=0$ and integrating along the segment from $y$ to $x$ gives
$|u(x)|\le CAa^2(1+\log(D/a))$. Since $a$ is comparable to $d(x)$,
we obtain
\[
|\nabla u(x)|\le CA\,d(x)\Lambda(x),\qquad
|u(x)|\le CA\,d(x)^2\Lambda(x).
\]

Next we cut off $u$ near the boundary. We use $\log\Lambda$ because
$\nabla(\log\Lambda)=-\nabla d/(d\Lambda)$ supplies the denominator
needed to cancel the $\Lambda$ in the boundary bounds.
Let $\eta:\R\to[0,1]$ be smooth, equal to one for $s\le1$
and zero for $s\ge2$, and set
\[
 \chi_j=\eta\!\left(\frac{\log\Lambda}{j}\right),\qquad j=1,2,\ldots.
\]
These cutoffs vanish where $\Lambda\ge e^{2j}$, hence in a neighborhood
of $\Omega^c$, and so belong to $C_c^\infty(\Omega)$.
Differentiating twice and using the bounds on $d$ and that $j,\Lambda \geq 1$ gives
\[
 |\nabla\chi_j|\le\frac{C}{jd\Lambda},\qquad
 |\Delta\chi_j|\le\frac{C}{jd^2\Lambda}.
\]
The product rule and the boundary bounds now give
\[
 |\Delta(\chi_ju)-\chi_j\Delta u|
 \le2|\nabla\chi_j||\nabla u|+|u||\Delta\chi_j|
 \le\frac{CA}{j}.
\]
Thus $\norm{\Delta(\chi_ju)}_\infty\le \norm{\Delta u}_\infty+CA/j$.
The set where $\chi_j\ne1$ lies in $\{\Lambda>e^j\}$, a layer that 
shrinks to $\Omega^c$. Since $u$ is continuous and vanishes there,
$\chi_ju\to u$ uniformly.
Finally, mollify $\chi_ju$ inside $\Omega$ with uniform error at most $1/j$, calling the
resulting function $u_j\in C_c^\infty(\Omega)$.
Mollification does not increase the $L^\infty$ norm of the Laplacian, and so these
functions satisfy \eqref{eq:smooth-approximation}.
\end{proof}
}

\section{Planar proof --- inequality in Theorem~\ref{thm:disk}}
\label{sec:planarproof}

Paired comparison~\cite{BL2026} decreases the quotient when we
pass from $\Omega$ to $\Omega^*$. Lemma~\ref{lem:smooth-trials}
allows us to start with smooth compactly supported trial functions. We set up the radial competitors
in $\R^n$ for use also in Theorem~\ref{thm:nd-infty}, then apply
the planar comparison. Equality cases are treated in
Section~\ref{sec:equality}.

Let $n\ge1$, let $\Omega\subset\R^n$ be nonempty, bounded and open,
and write $B=\Omega^*$. When we come to it, we will indicate the part of the proof that requires $n=2$. 

For an integrable source $f$ on $\Omega$, write $f^\downarrow$
for its signed decreasing rearrangement on $(0,1)$, equimeasurable
with $f$ when $\Omega$ is given normalized measure $|\Omega|^{-1}dx$.
On the ball $B$ with radius $R$,
the signed spherical rearrangement of $f$ is
$f^*(x)=f^\downarrow((|x|/R)^n)$.
These rearrangements preserve negative as well as positive values.

Fix $1\le p,q\le\infty$. We first take
$0\ne u\in C_c^\infty(\Omega)$ and set $f=-\Delta u$.
Then $f$ is bounded and
\[
 \int_\Omega f\,dx=-\int_\Omega\Delta u\,dx=0.
\]
Let $v,w$ be the radial Dirichlet solutions on $B$:
\begin{equation}\label{eq:ballproblems}
 \begin{cases}
 -\Delta v=f^*&\text{in }B,\\
 v=0&\text{on }\partial B,
 \end{cases}
 \qquad
 \begin{cases}
 -\Delta w=(-f)^*&\text{in }B,\\
 w=0&\text{on }\partial B.
 \end{cases}
\end{equation}
Since the sources in these equations are bounded, the maximum principle gives
$v,w\in L^\infty(B)$, hence $v,w\in L^p(B)$ for every
$1\le p\le\infty$.

The central step of the proof is now that our paired comparison result \cite[Theorem~1]{BL2026} applies  to this situation, giving when $n=2$ that either $v$ or $w$ has greater or equal $L^p$ norm than $u$:
\[
 \norm{u}_{L^p(\Omega)}
 \le M:=\max\{\norm{v}_{L^p(B)},\norm{w}_{L^p(B)}\} .
\]
Choose $z\in\{v,w\}$ that achieves the maximum, so that $\norm{z}_{L^p(B)}=M>0$.

The next two steps work in every dimension, provided $z$ has been obtained.

\subsubsection*{The case $1 \leq q < \infty$}

We must verify $z\in W_0^{2,q}(B)$ so that $z$ may be substituted as a trial function into the variational definition of $\Gamma_{p,q}(B)$. 

Let $\varphi\in L^r(B)$ be radial with
mean zero, for some $1<r<\infty$. The Dirichlet solution of  $-\Delta \zeta = \varphi$
belongs to $W^{2,r}(B)$ by~\cite[Theorem~9.15]{GilbargTrudinger2001}.
It satisfies 
\begin{equation}\label{eq:radial-flux}
 \zeta'(\varrho)=-\varrho^{1-n}\int_0^\varrho s^{n-1}\varphi(s)\,ds.
\end{equation}
Thus $\zeta$ is $C^1$ away from the origin, up to the boundary,
and has the clamped conditions $\zeta(R)=\zeta'(R)=0$. The trace characterization in Grisvard 
\cite[Corollary~1.5.1.6, p.~39]{Grisvard2011} yields 
$\zeta\in W_0^{2,r}(B)$.

Applying this fact to $\varphi=f^*$ and $(-f)^*$ gives that 
$z\in W_0^{2,q}(B)$ for every $1< q<\infty$, with $q=1$
following also by inclusion. Signed equimeasurability gives that $\norm{\Delta z}_{L^q(B)}
 =\norm{f}_{L^q(\Omega)} =\norm{\Delta u}_{L^q(\Omega)}$ 
and hence
\begin{align}
 \Gamma_{p,q}(B)
 &\le\frac{\norm{\Delta z}_{L^q(B)}}{\norm{z}_{L^p(B)}}
  =\frac{\norm{\Delta u}_{L^q(\Omega)}}{M}
  \le\frac{\norm{\Delta u}_{L^q(\Omega)}}
           {\norm{u}_{L^p(\Omega)}}.
 \label{eq:finiteq}
\end{align}
Taking the infimum over $u$ and recalling Lemma \ref{lem:smooth-trials} shows $\Gamma_{p,q}(B) \leq \Gamma_{p,q}(
\Omega)$.

\subsubsection*{The case $q=\infty$}

{
The preceding argument gives $z\in W_0^{2,r}(B)$
for every finite $r>1$. Since $\Delta z$ is bounded, we get 
$z\in W_C^{2,\infty}(B)$, and $z$ is an admissible trial function.
Signed equimeasurability therefore gives
\begin{align}
 \Gamma_{p,\infty}(B)
 &\le\frac{\norm{\Delta z}_{L^\infty(B)}}{\norm{z}_{L^p(B)}}
  =\frac{\norm{f}_{L^\infty(\Omega)}}{M}
  \le\frac{\norm{\Delta u}_{L^\infty(\Omega)}}
           {\norm{u}_{L^p(\Omega)}}.
 \label{eq:infiniteq}
\end{align}
Again, taking the infimum over $u$ and using Lemma \ref{lem:smooth-trials} proves \eqref{eq:main}.
}

\subsubsection*{Zero corner point $(p,q)=(\infty,1)$}

A logarithmic cutoff in two dimensions shows that $\Gamma_{\infty,1}(\Omega)=0$,
as follows. Choose $x_0\in\Omega$ and $R>0$ such that the closed ball 
$\overline{B_R(x_0)}$ is contained in $\Omega$. Fix a smooth function
$\eta:\R\to[0,1]$ with $\eta(t)=0$ for $t\le0$ and
$\eta(t)=1$ for $t\ge1$. For $0<\varepsilon<R$, write
$\tau_\varepsilon=\log(R/\varepsilon)$ and define
\[
 \chi_\varepsilon(x)=
 \eta\!\left(\frac{\log(R/|x-x_0|)}{\tau_\varepsilon}\right) , 
 \qquad \chi_\varepsilon(x_0)=1.
\]
Then $\chi_\varepsilon\in C_c^\infty(\Omega)$, equals $1$ on
$B_\varepsilon(x_0)$ and has $L^\infty$ norm $1$. With
$\varrho=|x-x_0|$ and $t=\log(R/\varrho)/\tau_\varepsilon$, the radial Laplacian is
\[
 \Delta \chi_\varepsilon(\varrho)
 =\frac{\eta''(t)}{\tau_\varepsilon^2\varrho^2} , 
 \qquad \varrho \in (\varepsilon,R),
\]
and vanishes outside this annulus. Consequently,
\[
 \frac{\norm{\Delta \chi_\varepsilon}_{L^1(\Omega)}}
      {\norm{\chi_\varepsilon}_{L^\infty(\Omega)}}
 =\frac{2\pi}{\tau_\varepsilon}
   \int_0^1|\eta''(t)|\,dt\longrightarrow0
\]
as $\varepsilon\to0$, proving the claim.
\qed

\hypertarget{remark:compliance}{}
\paragraph{Remark (connection with plate compliance).}

The case $(p,q)=(1,2)$ describes maximal plate compliance under
loads bounded by $1$, as we explain. For a bounded load 
$P\in L^\infty(\Omega)$, let $U_P\in H_0^2(\Omega)$ be the weak
solution of $\Delta^2U_P=P$, and put
$\mathcal C_{\rm plate}(\Omega,P)=\int_\Omega P U_P\,dx$.
Then 
\begin{equation}\label{eq:duality}
 \Gamma_{1,2}(\Omega)^{-2}
 =\sup_{\norm{P}_\infty\le1}
     \mathcal C_{\rm plate}(\Omega,P)
\end{equation}
by writing $\norm{u}_1=\sup_{\norm{P}_\infty\le1}
\int_\Omega P u\,dx$ in the quotient for $\Gamma_{1,2}(\Omega)^{-2}$ and interchanging
the two suprema.
Theorem~5.1 of Ashbaugh, Bucur, Laugesen and
Leylekian~\cite{ABLL2025} shows, in the plane, that every load
$|P|\le1$ satisfies
\[
 \mathcal C_{\rm plate}(\Omega,P)
 \le\mathcal C_{\rm plate}(\Omega^*,1).
\]
The constant load $1$ attains the disk supremum in
\eqref{eq:duality}, and so their theorem is equivalent to
\eqref{eq:main} at $(p,q)=(1,2)$.

\section{Higher dimensional inequalities and parameter ranges}
\label{sec:all-dimensions}

We prove the inequality for Theorem~\ref{thm:nd-infty} and determine
$\Gamma_{p,q}$ outside the compact embedding range shown in
Figure~\ref{fig:three-range}. Equality cases follow later in
Section~\ref{sec:nd-equality}.

\subsection*{Scaling and parameter ranges}

Let $n\ge3$ and let $\Omega\subset\R^n$ be nonempty, bounded and open.
In the $(1/p,1/q)$ square, the scaling law
\eqref{eq:scaling} gives the critical line
\begin{equation}\label{eq:ndline}
 \frac1q=\frac1p+\frac2n
\end{equation}
which runs from $(0,2/n)$ to $(1-2/n,1)$.
Below this line, where $1/q<1/p+2/n$, the trial spaces embed
compactly into $L^p$, as proved in Section~\ref{sec:attainment}. 
Proposition~\ref{prop:nd-noncompact} below treats this line and the region
above it.

\subsection*{Inequality in Theorem~\ref{thm:nd-infty} when $p=\infty$}

Take $0\ne u\in C_c^\infty(\Omega)$, set $f=-\Delta u$ and let
$v,w$ solve the rearranged equations in \eqref{eq:ballproblems} on the ball $B=\Omega^*$.
The key point is that our paired $L^\infty$ comparison \cite[Theorem~2]{BL2026} applies to precisely this set-up in all dimensions, giving that either $v$ or $w$ has greater or equal $L^\infty$ norm than $u$:
\[
 \norm{u}_{L^\infty(\Omega)}
 \le M:=\max\{\norm{v}_{L^\infty(B)},\norm{w}_{L^\infty(B)}\}.
\]
The clamped competitors from
\eqref{eq:finiteq}--\eqref{eq:infiniteq} therefore give
\begin{equation} \label{eq:rayleighCc}
 \Gamma_{\infty,q}(B)
 \le\frac{\norm{\Delta u}_{L^q(\Omega)}}{\norm{u}_{L^\infty(\Omega)}},
 \qquad 0\ne u\in C_c^\infty(\Omega).
\end{equation}
Taking the infimum over $C_c^\infty(\Omega)$ in
\eqref{eq:rayleighCc} and using Lemma~\ref{lem:smooth-trials}
proves \eqref{eq:nd-infty}.
\qed

\subsection*{Critical and supercritical exponents}

\begin{proposition}[Shape independence outside the compact range]
\label{prop:nd-noncompact}
Let $n\ge3$ and take $\Omega\subset\R^n$ to be nonempty, bounded
and open, with $1\le p,q\le\infty$. If $1/q>1/p+2/n$, or if
$(p,q)$ is an endpoint $(\infty,n/2)$ or $(n/(n-2),1)$ of
the critical segment \eqref{eq:ndline}, then $\Gamma_{p,q}(\Omega)=0$.
On the open critical segment, where $1/q=1/p+2/n$ and
$1<q<n/2$, the value is positive and independent of $\Omega$:
\[
 \Gamma_{p,q}(\Omega)
 =S_{n,q}:=
 \inf_{0\ne\phi\in C_c^\infty(\R^n)}
 \frac{\norm{\Delta\phi}_{L^q(\R^n)}}
      {\norm{\phi}_{L^p(\R^n)}}>0.
\]
\end{proposition}

\begin{proof}
For $1/q>1/p+2/n$, shrink a nonzero smooth bump into a ball
contained in $\Omega$. Its quotient scales as
$t^{n(1/q-1/p)-2}$ by \eqref{eq:scaling}; the exponent is
positive, and so the quotient tends to zero. Hence $\Gamma_{p,q}(\Omega)=0$.

On the open critical segment, where $1<q<n/2$ and
$p=nq/(n-2q)$, the whole-space Sobolev inequality
\cite[Theorem~2.3, p.~32]{GazzolaGrunauSweers2010} and
the Calder\'on--Zygmund estimate \eqref{eq:whole-space-cz} give, for every
$\phi\in C_c^\infty(\mathbb{R}^n)$,
\[
 \norm{\phi}_{L^p(\mathbb{R}^n)}
 \le C_{n,q}\norm{D^2\phi}_{L^q(\mathbb{R}^n)}
 \le C'_{n,q}\norm{\Delta\phi}_{L^q(\mathbb{R}^n)}.
\]
Hence $S_{n,q}\ge 1/C'_{n,q}>0$.
By Lemma~\ref{lem:smooth-trials}, $\Gamma_{p,q}(\Omega)\ge S_{n,q}$.
Conversely, the scaling exponent in \eqref{eq:scaling} is zero and so every nonzero $\phi\in C_c^\infty(\mathbb{R}^n)$ can
 be translated and shrunk into a ball contained
in $\Omega$ without changing its quotient.
Taking the infimum over these functions gives
$\Gamma_{p,q}(\Omega)\le S_{n,q}$.

At $(p,q)=(\infty,n/2)$, choose a ball $\overline{B_R(x_0)}\subset\Omega$
and a smooth $\eta:\R\to[0,1]$ equal to $0$ for $s\le0$
and $1$ for $s\ge1$. Set
\[
 \chi_\tau(x)=\eta\!\left(
   \frac{1}{\tau}\log\frac{R}{|x-x_0|}
 \right),
 \qquad \chi_\tau(x_0)=1,\qquad \tau\ge1.
\]
This function equals $1$ near $x_0$ and vanishes outside
$B_R(x_0)$, and so $\chi_\tau\in C_c^\infty(\Omega)$
with $\norm{\chi_\tau}_\infty=1$.
Radial differentiation on the transition annulus
$Re^{-\tau}<|x-x_0|<R$ gives
\[
 \norm{\Delta\chi_\tau}_{n/2}
 \le C \tau^{2/n-1}\longrightarrow0
\]
as $\tau \to \infty$. Hence $\Gamma_{\infty,n/2}(\Omega)=0$.

{
At the other endpoint of the segment, put $p=n/(n-2),q=1$, and suppose for the sake of obtaining a 
contradiction that $\Gamma_{p,1}(\Omega)>0$. Then
\[
 \norm{\phi}_p\le C_\Omega\norm{\Delta\phi}_1,
 \qquad \phi\in C_c^\infty(\Omega).
\]
Applying this estimate to differences shows that every Cauchy
sequence in the Laplacian norm is Cauchy in $L^p$.
Completeness of $L^p$ therefore extends the estimate to the
completion $W_{\Delta,0}^{2,1}(\Omega)$ of $C_c^\infty(\Omega)$
under that norm.
Since $L^p$ has the Fatou property, Fontana--Morpurgo's optimal
target theorem~\cite[Theorem~2, (55)]{FontanaMorpurgo2014}
would then imply $L^{p,\infty}(\Omega)\subset L^p(\Omega)$.
This inclusion is false: for any interior ball $B_R(x_0)$,
the function $|x-x_0|^{2-n}\mathbf{1}_{B_R(x_0)}(x)$ belongs
to weak $L^p$ but not to $L^p$, since its $p$-th power has
radial integral proportional to $\int_0^R d\varrho/\varrho=\infty$.
Hence $\Gamma_{p,1}(\Omega)=0$.
}
\end{proof}

\section{Interval proof --- inequality in Proposition~\ref{prop:interval}}
\label{sec:one-dimension}

We prove inequality \eqref{eq:interval-comparison} by closing the gaps between components. Write $\Omega=\bigcup_j I_j$ as a disjoint union of intervals of
lengths $L_j>0$, so that $\sum_j L_j=L$.
By Lemma~\ref{lem:smooth-trials}, it suffices to
consider $u\in C_c^\infty(\Omega)$. Only finitely many components meet its
support. Translate these components consecutively into $(0,L)$ and
put $u=0$ on the unused part. The resulting function belongs to
$C_c^\infty(0,L)$ and has the same $L^p$ and second-derivative $L^q$
norms as $u$. Taking the infimum with respect to $u$ proves $\Gamma_{p,q}(\Omega)\ge\Gamma_{p,q}(\Omega^*)$. 
The equality case is treated in Section~\ref{sec:interval-equality}.
\qed

\section{Positivity and attainment --- proof of Proposition~\ref{prop:attainment}}
\label{sec:attainment}

We prove compact embeddings, positivity and attainment, and then treat
endpoint nonattainment.

\subsubsection*{Part (\ref{part:attainment-compactness}): compact embeddings}
{
At $q=\infty$, a sequence bounded in $\norm{\Delta u}_\infty$
is bounded in $W^{2,r}$ for every finite $r>1$, by
\eqref{eq:wc-bound}. Choose $r>n$. Zero extension to a fixed
enclosing ball and compact Sobolev embedding give a uniformly
convergent subsequence, and hence convergence in every $L^p$.
This proves the compact embedding of $W_C^{2,\infty}$ into $L^p$ in every
dimension. In the remaining paragraphs, assume $q<\infty$.
}

In dimension one, zero extension to a fixed enclosing interval $I$,
H\"older's inequality and the compact embedding
$W^{2,1}(I)\hookrightarrow C(\overline I)$ give a compact embedding
$W_0^{2,q}(\Omega)\hookrightarrow L^p(\Omega)$ for every
$1\le p\le\infty, 1\le q<\infty$.

In dimension two, $W_0^{2,q}(\Omega)$ is contained in the Sobolev space
of a fixed enclosing disk. For $q>1$, Sobolev--Morrey and Rellich
theorems~\cite[Chapter~7]{GilbargTrudinger2001} give compact embeddings
into $C^0$ and hence every $L^p$.
For $q=1$, note the proposition requires $1\le p<\infty$.
The fundamental theorem of calculus applied twice gives
\[
 \norm{u}_\infty\le\norm{\partial_{12}u}_1 ,
 \qquad u\in C_c^\infty(\Omega),
\]
and density extends this estimate to $W_0^{2,1}$.
The compact embedding of $W^{2,1}_0$ into $L^1$ from Rellich's theorem and
interpolation with the preceding bound give compact embeddings into $L^p$ for every finite $p$.

{
For $n\ge3$, zero extension to a fixed enclosing ball $B$ and
the Sobolev--Morrey and Rellich theorems
\cite[Chapter~7]{GilbargTrudinger2001}
give compact embeddings into $L^p(\Omega)$ for finite $p$ whenever
$1/q<1/p+2/n$, and into $C^0(\overline B)$ and hence
$L^\infty(\Omega)$ when $q>n/2$.
}

\subsubsection*{Part (\ref{part:attainment-compactness}): positivity of the infimum}
We prove $\norm{u}_p\le C\norm{\Delta u}_q$, so that
$\Gamma_{p,q}(\Omega)\ge C^{-1}>0$. 

For $q=\infty$, this bound follows
from \eqref{eq:wc-bound} and Sobolev embedding at a fixed $r>n$.

Suppose $q<\infty$. In dimension $n=1$, the required bound follows from the estimate $\norm{u}_{L^\infty(I)}\le C\norm{u''}_{L^1(I)}$. For $n\ge2$, extend $u\in C_c^\infty(\Omega)$
by zero and fix $D>\operatorname{diam}\Omega$. With $\Phi$ the fundamental
solution of $-\Delta$, the representation formula
\cite[formula~(2.17), p.~18]{GilbargTrudinger2001} gives
$u=\Phi*(-\Delta u)$ and hence $|u|\le K_D*|\Delta u|$ on $\Omega$, with
\[
 K_D(z)=\mathbf1_{\{|z|<D\}}
 \begin{cases}
 (2\pi)^{-1}\log(D/|z|),&n=2,\\
 \kappa_n|z|^{2-n},&n\ge3.
 \end{cases}
\]
In the plane,
the added constant in the logarithm cancels because
$\int_\Omega\Delta u\,dx=0$.
The kernel belongs to $L^r$ for every finite $r$ when $n=2$, and
for $1\le r<n/(n-2)$ when $n\ge3$. If $p \geq q$, take
$1/r=1+1/p-1/q \leq 1$. The hypothesis \(1/q<1/p+2/n\) gives \(1/r>(n-2)/n\), and hence \(K_D\in L^r\). Young's convolution inequality now gives $\norm{u}_{L^p(\Omega)}
\le C   \norm{\Delta u}_{L^q(\Omega)}$ with $C=\norm{K_D}_{L^r(\mathbb R^n)}$. 
The case $p<q$ follows by H\"{o}lder. Density extends the estimate to $u \in W_0^{2,q}(\Omega)$.

\subsubsection*{Part (\ref{part:attainment-existence}): attainment for $1<q\le\infty$}

First suppose $1<q<\infty$. Normalize an infimizing
sequence by $\norm{u_j}_p=1$. Since $1<q<\infty$, the Calder\'on--Zygmund estimate 
\eqref{eq:whole-space-cz} and Poincar\'e inequalities on a fixed
enclosing ball give $ \norm{u_j}_{W^{2,q}(\Omega)}
 \le C\norm{\Delta u_j}_{L^q(\Omega)}$. Reflexivity yields
a weakly convergent subsequence in $W_0^{2,q}(\Omega)$; compact
embedding gives strong convergence in $L^p$ for all $p$ covered by the hypotheses of the proposition. The limit has $L^p$ norm $1$, and weak lower
semicontinuity of the Laplacian norm proves attainment.

{
For $q=\infty$, normalize a minimizing sequence in
$W_C^{2,\infty}(\Omega)$ by $\norm{u_j}_p=1$.
Equation~\eqref{eq:wc-bound} and part~(\ref{part:attainment-compactness})
give a subsequence converging uniformly in $L^\infty$ and also weakly in $W^{2,r}$
at every integer exponent $r\ge2$.
Weak closure of $W_0^{2,r}$ and inclusion between finite exponents
put the limit $u$ in every $W_0^{2,r}$, $1<r<\infty$.
A weak-star limit of the Laplacians equals $\Delta u$
distributionally, with
$\norm{\Delta u}_\infty\le\liminf_j\norm{\Delta u_j}_\infty$.
Uniform convergence preserves $\norm{u}_p=1$, including when $p=\infty$.
Thus $u\in W_C^{2,\infty}(\Omega)$ attains the infimum.
}

\subsubsection*{Part (\ref{part:attainment-finite-endpoints}): nonattainment for $q=1$ and $1 \leq p < \infty$}

Suppose that $u$ attains
$\gamma=\Gamma_{p,1}(\Omega)>0$, normalized by $\norm{u}_p=1$.
Let $p'$ be the H\"older conjugate of $p$ and choose
$g\in L^{p'}(\Omega)$ with $\norm{g}_{p'}=1$ and
$\int_\Omega gu\,dx=1$. Note from part~(\ref{part:attainment-compactness}) that
$W_0^{2,1}(\Omega)\hookrightarrow L^p(\Omega)$ under the hypotheses of the proposition,
and so every $\phi\in W_0^{2,1}(\Omega)$ is an admissible trial function. 

For every $\phi\in W_0^{2,1}(\Omega)$,
\[
 \gamma\left|\int_\Omega g\phi\,dx\right|
 \le\norm{\Delta\phi}_1.
\]
The map $-\Delta\phi\mapsto\gamma\int_\Omega g\phi\,dx$ is therefore a bounded
functional on a subspace of $L^1(\Omega)$. Hahn--Banach gives
$\xi\in L^\infty(\Omega)$ with $\norm{\xi}_\infty\le1$ and
\[
 -\int_\Omega\xi\Delta\phi\,dx
 =\gamma\int_\Omega g\phi\,dx
 \quad\text{for all }\phi\in W_0^{2,1}(\Omega).
\]
Equality at $\phi=u$ forces $\xi=\operatorname{sgn}(-\Delta u)$
almost everywhere on $E=\{\Delta u\ne0\}$. Distributionally,
$-\Delta\xi=\gamma g$. Interior elliptic regularity gives
$\xi\in W_{\rm loc}^{2,p'}$ when $p>1$; when $p=1$, it gives
$\xi\in W_{\rm loc}^{2,r}$ for every finite $r>1$.
A Sobolev function has gradient zero almost everywhere on each of
its level sets~\cite[Lemma~7.7, p.~152]{GilbargTrudinger2001}.
Applying this fact first to $\xi$ and then to its first derivatives
gives $D^2\xi=0$ almost everywhere on $\{\xi=\pm1\}$.
Thus $-\Delta\xi=\gamma g=0$
almost everywhere on $E$, and our choice of $g$ implies $u=0$
there. The same two applications of the level set property to
$u\in W^{2,1}$ give $\Delta u=0$ almost everywhere on $E$, a
contradiction because then $\gamma=0$.

\subsubsection*{Part (\ref{part:attainment-one-dimensional-q1}), $n=1$: nonattainment at $(p,q)=(\infty,1)$}
{
Attainment of
$\Gamma_{\infty,1}(I)=8/|I|$ for an interval is precluded by the strict estimate
$\norm{v''}_1>(8/|I|)\norm{v}_\infty$ that was proved in
Appendix~\ref{app:interval-pinfty}. 

For general $\Omega$, the zero extension of any
$0\ne u\in W_0^{2,1}(\Omega)$ is continuous with compact support,
and so $|u|$ reaches its positive maximum on some component $I$.
Restricting smooth approximations gives $u|_I\in W_0^{2,1}(I)$,
while extending interval trial functions by zero gives
$\Gamma_{\infty,1}(I)\ge\Gamma_{\infty,1}(\Omega)$. Hence
\[
 \frac{\norm{u''}_{L^1(\Omega)}}{\norm{u}_{L^\infty(\Omega)}}
 \ge\frac{\norm{u''}_{L^1(I)}}{\norm{u}_{L^\infty(I)}}
 >\Gamma_{\infty,1}(I)
 \ge\Gamma_{\infty,1}(\Omega),
\]
and so $u$ cannot attain the infimum.
}

\section{Equality cases: disks, balls and intervals}
\label{sec:equality}

We complete Theorems~\ref{thm:disk} and~\ref{thm:nd-infty}
and Proposition~\ref{prop:interval} by identifying the equality cases.
The paired comparison equality cases from \cite{BL2026} handle
disks and balls; capacity cutoffs explain the planar $q=1$
exception, and component estimates settle intervals.

\subsection{Equality for clamped functions via paired comparison}
\label{sec:comparison-equality}

The companion paper's Theorems~1 and~2 in \cite{BL2026} identify the domain up to
a set of zero $H^1$-capacity and the solution up to translation.
For our clamped functions, continuity excludes the exceptional set.
The following lemma extends the comparison to the source exponents
needed here and rules out the two-disk alternative at the planar
$L^1$ endpoint.

\begin{lemma}[Equality for clamped paired comparison]
\label{lem:clamped-equality}
Let $\Omega\subset\R^n$ be nonempty, bounded and open. Assume $n=2$ and
$1\le p\le\infty$, or else $n \geq 3$ and
$p=\infty$. Suppose $n/2<r<\infty$ and $0\ne u\in W_0^{2,r}(\Omega)$. If 
$f=-\Delta u$ and $v,w$ solve \eqref{eq:ballproblems} on the equal-volume ball $B$, then
\[
 \norm{u}_{L^p(\Omega)}
 \le\max\{\norm{v}_{L^p(B)},\norm{w}_{L^p(B)}\}.
\]
If equality holds, then $\Omega=c+B$ for some $c\in\R^n$ and
$u(x)=v(x-c)$ or $u(x)=-w(x-c)$.
\end{lemma}

\begin{proof}
{
The hypotheses imply $u\in H_0^1(\Omega)\cap L^\infty(\Omega)$
and $\int_\Omega f\,dx=0$.
Although the theorems in \cite{BL2026} are stated for $L^2$
sources, their proofs apply here for the following reasons.

Since $r>n/2$, H\"older's inequality and Sobolev embedding show
that $\phi\mapsto\int_\Omega f\phi\,dx$ is continuous on
$H_0^1(\Omega)$, and elliptic regularity gives
$v,w\in W^{2,r}(B)\subset L^\infty(B)$.
Thus the positive and negative parts of $u$, and their level set
truncations, are valid test functions in the weak equations. For the key equality argument in the proof of \cite[Lemma~7]{BL2026}, take $(h,g)=(u_+,f)$ or $(u_-,-f)$ and note $f \not \equiv 0$ since $u$ is not the zero function. Testing the equation for $u$ or
$-u$, respectively, with $h$ and applying Hardy--Littlewood give
\[
 \int_\Omega|\nabla h|^2\,dx
 =\int_\Omega gh\,dx
 \le\int_B g^*h^*\,dx.
\]
Both integrals on the right are absolutely convergent,
since $g,g^*$ are integrable and $h,h^*$ are bounded.

The remaining steps of the equality proofs in \cite[Sections 5 and 7]{BL2026} carry through since the primitive $F(s)=\int_0^s f^\downarrow(t)\,dt$ is continuous and concave,
vanishes at both endpoints and is positive on $(0,1)$. Thus the equality conclusions of \cite[Theorems~1 and~2]{BL2026} remain valid in the current situation.
}

Now suppose equality holds in the conclusion of
Lemma~\ref{lem:clamped-equality}. For $n=2$ and $1<p<\infty$, apply the equality statement of
\cite[Theorem~1]{BL2026}, and for $n \geq 2$ and $p=\infty$ apply
\cite[Theorem~2]{BL2026}. After translating $\Omega$ and $u$, these statements identify
$\Omega$ with $B$, possibly with an exceptional set of zero
$H^1$-capacity removed, and identify $u$ with $v$ or $-w$
almost everywhere. The zero extension of $u$ is continuous on $\R^n$ and vanishes at every point outside $\Omega$, by Sobolev--Morrey embedding and uniform
convergence of smooth approximations, since $2r>n$. By continuity, the zero extension of $u$ agrees with $v$ or
$-w$ at every point of $B$. Since $v,w>0$ in $B$, the
exceptional set must be empty. Thus $\Omega$ is a ball. 

At $(n,p)=(2,1)$, the same argument handles the one-ball case of the equality statement of \cite[Theorem~1]{BL2026}.
The remaining possibility there is that 
$\Omega=D_+\cup D_-$ is a union of two disjoint disks,
with $u$ given by translated truncated radial potentials of opposite
signs (after excluding an exceptional set by continuity). 
The radial Poisson equation for the positive truncated potential gives $\int_{D_+}f\,dx=|\Omega|F(|D_+|/|\Omega|)>0$. 
On the other hand, $u|_{D_+}\in W_0^{2,r}(D_+)$ and so
$\int_{D_+}f\,dx=0$. This contradiction excludes the two-disk
alternative and completes the equality proof.
\end{proof}

\subsection{Equality in Theorem~\ref{thm:disk} for \texorpdfstring{$q>1$}{q > 1}}
\label{sec:planar-equality}

Let $n=2$ and $1<q \leq \infty$. 
Let $B=B_R$ be the equal-area disk and suppose
$\Gamma_{p,q}(\Omega)=\Gamma_{p,q}(B)$.

Proposition~\ref{prop:attainment}(\ref{part:attainment-existence})
supplies a minimizer $u$ for $\Gamma_{p,q}(\Omega)$ in \eqref{eq:gamma}, which we may normalize by $\norm{u}_p=1$. Put $f=-\Delta u$, let $v,w$ solve
\eqref{eq:ballproblems} and set $M=\max\{\norm{v}_p,\norm{w}_p\}$.
Choose $z\in\{v,w\}$ with $\norm{z}_p=M$.
For $1<q<\infty$, elliptic regularity gives
$v,w\in W^{2,q}(B)$. Their boundary values vanish, and
the mean-zero sources $f^*$ and $(-f)^*$ give vanishing normal derivatives. The trace characterization used after formula \eqref{eq:radial-flux} therefore gives $v,w\in W_0^{2,q}(B)$.
When $q=\infty$, the same argument applies at every finite
exponent $r>1$, and since their negative Laplacians $f^*$ and $(-f)^*$ are bounded, we get 
$v,w\in W_C^{2,\infty}(B)$.
For all $q$, Sobolev embedding also gives $v,w\in L^\infty(B)$,
hence membership in $L^p(B)$ for every $1\le p\le\infty$.
Thus $z$ is an admissible trial function for $\Gamma_{p,q}(B)$.

Signed equimeasurability and the inequality in Lemma~\ref{lem:clamped-equality} now give
\[
 \Gamma_{p,q}(B)
 \le\frac{\norm{f}_q}{M}
 \le\frac{\norm{f}_q}{\norm{u}_p}
 =\Gamma_{p,q}(\Omega).
\]
Equality holds throughout and so $\norm{u}_p=M$.
The equality statement in Lemma~\ref{lem:clamped-equality}, in which we take $r=q$ when $1<q<\infty$
and take any finite $r>1$ when $q=\infty$, now yields that $\Omega$ is a disk.

Conversely, if $\Omega$ is a disk then equality holds in the theorem by translation invariance.
\qed

\subsection{The \texorpdfstring{$q=1$}{q = 1} exception}
\label{sec:q1-equality}

At $(p,q)=(\infty,1)$, the logarithmic construction in
Section~\ref{sec:planarproof} gives that $\Gamma_{\infty,1}=0$ on every set, so that both sides of \eqref{eq:main} equal zero. 

Now suppose $q=1$ and $1 \leq p < \infty$. In this case equality holds in the theorem if $\Omega$ is a disk with an exceptional set of capacity zero (and hence measure zero) removed, by the next result. We do not claim ``only if''. 

\begin{proposition}[Removing sets of zero capacity]
\label{prop:capacity-removal}
If $G\subset\R^2$ is a nonempty bounded open set and
$K\subset G$ has zero $H^1$-capacity with $G\setminus K$ open, then
\[
\Gamma_{p,1}(G) = \Gamma_{p,1}(G\setminus K) , \qquad 1\le p<\infty .
\]
\end{proposition}

\begin{proof}
Fix $0\ne u\in C_c^\infty(G)$ and a smooth open set
$G_0\Subset G$ containing $\operatorname{supp}u$.
The cutoff lemma of Brezis, Marcus and Ponce
\cite[Lemma~4.E.1, p.~104]{BrezisMarcusPonce2007}
gives $\chi_j\in C_c^\infty(G_0)$, equal to $1$ near the compact set
$K\cap\operatorname{supp}u$, with $0\le\chi_j\le1$ and
$\norm{\Delta\chi_j}_1\to0$. Integration by parts gives
\[
 \norm{\nabla\chi_j}_2^2
 =-\int_{G_0}\chi_j\Delta\chi_j\,dx
 \le\norm{\Delta\chi_j}_1\longrightarrow0.
\]
Hence $\norm{\chi_j}_2\to0$ and 
$\norm{\nabla\chi_j}_1\to0$ and, after passing to a subsequence,
$\chi_j\to0$ almost everywhere.
Letting $u_j=(1-\chi_j)u\in C_c^\infty(G\setminus K)$, we have 
$u_j\to u$ in $L^p$ and $\Delta u_j\to\Delta u$ in $L^1$
by dominated convergence and the product rule.
Lemma~\ref{lem:smooth-trials} now gives
$\Gamma_{p,1}(G\setminus K)\le\Gamma_{p,1}(G)$.
The reverse inequality is immediate by inclusion.
\end{proof}

\subsection{Equality in Theorem~\ref{thm:nd-infty}}
\label{sec:nd-equality}

{
The proof in Section~\ref{sec:planar-equality} continues to apply with
$p=\infty$ and $n/2<q\le\infty$, for $n \geq 3$.
Proposition~\ref{prop:attainment}(\ref{part:attainment-existence})
supplies a minimizer for every such $q$. The equality $\Gamma_{\infty,q}(\Omega)=\Gamma_{\infty,q}(B)$ forces
equality in Lemma~\ref{lem:clamped-equality}, using $r=q$ for
finite $q$ and any finite $r>n/2$ for $q=\infty$.
Hence $\Omega$ is a ball. Conversely, a ball gives equality
by translation invariance.
\qed
}

\subsection{Equality in Proposition~\ref{prop:interval}}
\label{sec:interval-equality}

Let $L=|\Omega|$, write the component lengths of $\Omega$ as $L_j$
and put $L_{\max}=\max_j L_j$.
If there is only one component then $\Omega$ is an interval and
there is nothing to prove.

Suppose $\Omega$ has more than one component, write these intervals
as $I_j$ and put $C_{p,q}=\Gamma_{p,q}((0,1))$ and
$\sigma=2+1/p-1/q>0$. For a nonzero admissible trial function $u$, set
\[
 a_j=\norm{u}_{L^p(I_j)},\qquad
 b_j=\norm{u''}_{L^q(I_j)}.
\]
The restriction of $u$ to each component is clamped, and the scaling
law \eqref{eq:scaling} gives
$C_{p,q} a_j\le L_j^\sigma b_j$.
Also, $\norm{u}_p=\norm{(a_j)}_{\ell^p}$ and
$\norm{u''}_q=\norm{(b_j)}_{\ell^q}$.

If $p\ge q$, then $\norm{(a_j)}_{\ell^p}\le\norm{(a_j)}_{\ell^q}$
and so
\[
C_{p,q}  \norm{u}_p
 \le C_{p,q} \norm{(a_j)}_{\ell^q}
 \le L_{\max}^{\sigma}\norm{(b_j)}_{\ell^q}
 = L_{\max}^{\sigma}\norm{u''}_q.
\]
If $p<q$, define $s$ by $1/s=1/p-1/q>0$.
H\"older's inequality for sequences gives
\[
C_{p,q} \norm{u}_p
 \le \norm{(L_j^\sigma b_j)}_{\ell^p}
 \le 
       \norm{(L_j^\sigma)}_{\ell^s} \norm{(b_j)}_{\ell^q} = 
       \norm{(L_j^\sigma)}_{\ell^s} \norm{u''}_q.
\]
Using that $\sigma s=1+2s$, rearranging the estimates and taking the
infimum over $u$ yields
\[
 \Gamma_{p,q}(\Omega)\ge C_{p,q}
 \begin{cases}
 L_{\max}^{-\sigma},&p\ge q,\\[2pt]
 \left(\sum_j L_j^{1+2s}\right)^{\!-1/s},
   &p<q.
 \end{cases}
\]
Since $\Omega$ has more than one component, $L_{\max}<L$ and $\sum_j L_j^{1+2s}<L^{1+2s}$. Hence $\Gamma_{p,q}(\Omega)>C_{p,q}L^{-\sigma}
=\Gamma_{p,q}((0,L))$, giving strict inequality in the proposition.
\qed

\appendix
{
\section{Explicit values for the interval}
\label{app:interval-values}

We derive the four edge formulas in Table~\ref{tab:interval-values}
and the classical clamped beam value stated after
Proposition~\ref{prop:interval}. Unless otherwise stated, $I=(0,L)$.

\subsection{The edge \texorpdfstring{$p=\infty$}{p = infinity}}
\label{app:interval-pinfty}

We give an elementary proof of the formula obtained by
Oshime~\cite{Oshime2008} and reproved by
Watanabe \emph{et al.}~\cite{WatanabeEtAl2009}, also treating
the endpoints in Table~\ref{tab:interval-values}.
For $1\le q<\infty$, take $0\ne u\in W_0^{2,q}(I)$;
for $q=\infty$, take $0\ne u\in W_C^{2,\infty}(I)$.
After changing sign, choose $a\in I$ with $u(a)=\norm{u}_\infty=M>0$.
Since $u'(0)=u'(a)=0$, integration by parts and H\"older give
\[
 M=-\int_0^a(x-a/2)u''(x)\,dx
 \le\frac{a^{2-1/q}}{2(q'+1)^{1/q'}}\norm{u''}_{L^q(0,a)}.
\]
For $q<\infty$, combining with the analogous estimate on $(a,L)$
and using convexity yields
\[
 \norm{u''}_q
 \ge 2(q'+1)^{1/q'}M
      \left(a^{1-2q}+(L-a)^{1-2q}\right)^{\!1/q}
 \ge\frac{8(q'+1)^{1/q'}}{L^{2-1/q}}M.
\]
For $q=\infty$, the two estimates instead give
\[
 \norm{u''}_\infty
 \ge\frac{4M}{\min\{a^2,(L-a)^2\}}
 \ge\frac{16M}{L^2}.
\]
For $1<q\le\infty$, equality holds at $a=L/2$ for the clamped
function determined by
\[
 u''(x)=\operatorname{sgn}(L/4-x)|L/4-x|^{q'-1},
 \qquad 0<x<L/2,\qquad u(0)=u'(0)=0,
\]
and even reflection about $L/2$, using only the sign when $q=\infty$.
Oddness of $u''$ about $L/4$ gives $u'(L/2)=0$, and this choice
gives equality in H\"older.
When $q=1$, the first estimate is strict because
$|x-a/2|<a/2$ almost everywhere on $(0,a)$ and $u''$ is not zero
there. Thus $\norm{u''}_1>(8/L)\norm{u}_\infty$, proving
nonattainment. Sharpness follows from the tent sequence in
Appendix~\ref{app:interval-qone}.
\qed

\subsection{The edge \texorpdfstring{$q=1$}{q = 1}}
\label{app:interval-qone}

We prove the $q=1$ row of Table~\ref{tab:interval-values}, whose case $p=1$
is due to Parini--Ruf--Tarsi~\cite{PariniRufTarsi2015}.
For $0\ne u\in W_0^{2,1}(I)$, put $A=\norm{u''}_1$,
$\alpha=\max u'>0$ and $\beta=-\min u'>0$.
Since $u'$ starts and ends at zero, its total variation satisfies
$A\ge2(\alpha+\beta)$.
On a positive component $(a_j,b_j)$ of $\{u\ne0\}$, we have
\[
 u(x)\le\min\{\alpha(x-a_j),\beta(b_j-x)\},
\]
and on a negative component the same bound holds for $-u$ with
$\alpha,\beta$ interchanged. Each bounding triangle has base
$L_j=b_j-a_j$ and height $\alpha\beta L_j/(\alpha+\beta)$.
Integrating their $p$-th powers and summing gives
\[
 \norm{u}_p^p
 \le\frac1{p+1}\left(\frac{\alpha\beta}{\alpha+\beta}\right)^{\!p}
       \sum_j L_j^{p+1}
 \le\frac{A^pL^{p+1}}{8^p(p+1)},
\]
using $\alpha\beta\le(\alpha+\beta)^2/4$ and $\sum_jL_j\le L$.
These bounds are sharp: the symmetric height-one tent supported on
$(\varepsilon,L-\varepsilon)$ has second derivative of total variation
$8/(L-2\varepsilon)$. Mollifying inside $I$ and letting
$\varepsilon\downarrow0$ gives smooth trial functions with second-derivative
$L^1$ norm tending to $8/L$ and $L^p$ norm tending to
$(L/(p+1))^{1/p}$ for finite $p$, or $1$ for $p=\infty$.
This proves the entire $q=1$ row, including $p=\infty$;
nonattainment at that corner was proved in
Appendix~\ref{app:interval-pinfty}.
\qed

\subsection{The edge \texorpdfstring{$q=\infty$}{q = infinity}}
\label{app:interval-qinfty}

The profile of Katzourakis--Parini~\cite[Remark~3.7]{KatzourakisParini2017},
scaled to have $\norm{U''}_\infty=1$, is
\[
 U(x)=\begin{cases}
 x^2/2,&0\le x\le L/4,\\
 L^2/16-(x-L/2)^2/2,&L/4\le x\le3L/4,\\
 (L-x)^2/2,&3L/4\le x\le L.
 \end{cases}
\]
Take $0\ne u\in W_C^{2,\infty}(I)$ and normalize
$\norm{u''}_\infty=1$. On each outer quarter, $u-U$ is
concave with zero value and derivative at the adjacent endpoint, and so
$u-U\le0$. On the middle half it is convex with nonpositive endpoint
values, and hence is again nonpositive. Applying the same argument to
$-u$ gives $|u|\le U$. Thus $\Gamma_{p,\infty}(I)=\norm{U}_p^{-1}$.
Direct integration gives $\norm{U}_p=L^{2+1/p}/(16c_p)$,
including $p=\infty$ with $c_\infty=1$. This proves the $q=\infty$
row of Table~\ref{tab:interval-values}.
\qed

\subsection{The edge \texorpdfstring{$p=1$}{p = 1}}
\label{app:interval-pone}

We give a dual
form of the interpolation argument underlying the kernel proof for $(p,q)=(1,2)$ by Hindov--Nitzan--Olsen--Rydhe~\cite{HindovEtAl2025}. 
Work on $I=(-1,1)$ so that $L=2$, and write $r=q'$ so that $1/r+1/q=1$. For the formula in Table~\ref{tab:interval-values} we must show $\Gamma_{1,q}(I)=2^{1-1/r}m_r$, after which the scaling law \eqref{eq:scaling} produces the general case. 

For $1\le q<\infty$, take $0\ne u\in W_0^{2,q}(I)$;
for $q=\infty$, take $0\ne u\in W_C^{2,\infty}(I)$.
Choose $\phi\in C^\infty(\overline I)$ with $|\phi|\le1$,
and let $\psi$ solve $-\psi''=\phi$, with conditions 
$\psi(0)=\psi'(0)=0$ imposed for definiteness. For $0<a<1$, let $\ell_a$ be the affine
interpolant of $\psi$ at $\pm\sqrt a$. The interpolation
remainder gives $|\psi(t)-\ell_a(t)|\le|t^2-a|/2$ on $[-1,1]$.
Integration by parts using the clamped conditions yields that 
\[
 \int_{-1}^1u(t)\phi(t)\,dt
 =-\int_{-1}^1u''(t)(\psi(t)-\ell_a(t))\,dt
 \le\frac12\norm{t^2-a}_{L^r(-1,1)}\norm{u''}_q.
\]
Taking the supremum over $\phi$, then the minimum over $a$, gives
$\norm{u}_1\le 2^{1/r-1} m_r^{-1} \norm{u''}_q$, so that $\Gamma_{1,q}(I) \geq 2^{1-1/r}m_r$. 

To prove equality when $1<q<\infty$, choose $a \in (0,1)$ to minimize
$\norm{t^2-a}_{L^r(-1,1)}$, so that the function 
$f(t)=\operatorname{sgn}(a-t^2)|t^2-a|^{r-1}$ has $\int_{-1}^1f(t)\,dt=0$.
Evenness also gives $\int_{-1}^1tf(t)\,dt=0$. Thus
\[
 U(x)=-\int_{-1}^x(x-t)f(t)\,dt
\]
is clamped at both endpoints and has $-U''=f$.
Evenness and zero mean give $U'(0)=-\int_{-1}^0f(t)\,dt=0$
and make $U'$ odd. Since $U''=-f$ is positive on $(-1,-\sqrt a)$
and negative on $(-\sqrt a,0)$, the conditions $U'(-1)=U'(0)=0$
give $U'>0$ on $(-1,0)$. Oddness gives $U'<0$ on $(0,1)$,
and the zero endpoint values of $U$ then imply $U>0$ inside.
Next, two integrations by parts using $f=-U''$
and $U(\pm1)=U'(\pm1)=0$ give $\frac12\int_{-1}^1(a-t^2)f(t)\,dt=\int_{-1}^1 U(t) \, dt = \norm{U}_1$. 
Our choice of $f(t)$ has the same sign as $a-t^2$ and so the definition of $f$ gives  
\[
(a-t^2)f(t) = |a-t^2|^r = |f(t)|^q ,
\]
using that $r=(r-1)q$. Thus  
\[
 \norm{U}_1=\frac12\int_{-1}^1(a-t^2)f(t)\,dt
 =\frac12 \norm{a-t^2}_{L^r(-1,1)} \norm{f}_{L^q(-1,1)}.
\]
Consequently $\Gamma_{1,q}(I)=2^{1-1/r}m_r$ as wanted. 

At the endpoints $q=1$ and $q=\infty$, one finds $m_\infty=2$ and $m_1=4$,
and so the formulas are consistent with those endpoints of Appendices~\ref{app:interval-qone}
and~\ref{app:interval-qinfty}.
For $q=2$, the minimizing constant is $a=1/3$ and
$m_2=3\sqrt5/2$, recovering the previously stated value
$\Gamma_{1,2}(I)=12\sqrt5/L^{5/2}$.
\qed

\subsection{The vibrating clamped beam}

On $I=(-L/2,L/2)$, the classical ground state $u\in H_0^2(I)$ is
\[
 u(x)=\cosh(\beta_1/2)\cos(\beta_1x/L)
      -\cos(\beta_1/2)\cosh(\beta_1x/L),
\]
where $\beta_1$ is the least positive root of
$\cos\beta\cosh\beta=1$.
Since $u''''=(\beta_1/L)^4u$, the Rayleigh principle gives
$\Gamma_{2,2}(I)=\beta_1^2/L^2$.
\qed
}

\section*{Funding acknowledgment}

Richard Laugesen's research was supported by grants from the
Simons Foundation (\#964018) and the National Science Foundation
(\#2246537).
Laugesen is grateful to the Laboratoire de Mathématiques, Université
Savoie Mont Blanc, for generous travel support.

Dorin Bucur was partially supported by the ANR project STOIQUES financed by
the French Agence Nationale de la Recherche (ANR-24-CE40-2216).

\section*{Acknowledgment of AI assistance}

The authors used OpenAI's ChatGPT through Codex to explore the theorems and 
conjectures, check various endpoint arguments, locate and verify
references, and prepare this manuscript and its figures.
Responsibility for the mathematical statements and their verification
rests with the authors.

\end{document}